\documentclass[12pt]{amsart}
\usepackage{graphicx}
\usepackage{amsmath, amsfonts,multirow}
\usepackage{amscd}
\usepackage{graphics}
\usepackage{latexsym}

\newtheorem{thm}{Theorem}[section]
\theoremstyle{definition}
\newtheorem{cor}[thm]{Corollary}

\newtheorem{prop}[thm]{Proposition}

\newtheorem{theo}[thm]{Theorem}
\newtheorem{lemm}[thm]{Lemma}

\newtheorem{rem}[thm]{Remark}
\newtheorem{exam}[thm]{Example}
\numberwithin{equation}{section}
\begin{document}

% Enter full title and short title for running headers
\title{Contact Structures on Plumbed 3-Manifolds}

% Author name(s)
\author{\c{C}a\u{g}r\i \, Karakurt}%\affil{1}}
% Abbreviated author name for running headers
\email{ karakurt@math.msu.edu}%
\date{\today}
% Abbreviated author name for first page header
%\headabbrevauthor{Author, F., and S. Author}

\address{Department of Mathematics, The University of Texas at Austin, 1 University Station, C1200 TX 78712.}
%and
%\affilnum{2}Complete Second Author Address}

% Address / e-mail address of corresponding author

\begin{abstract}

In this paper, we show that the Ozsv\'ath-Szab\'o contact invariant $c^+(\xi)\in HF^+(-Y)$ of a contact $3$-manifold $(Y,\xi)$ can be calculated combinatorially if $Y$ is the boundary of a certain type of plumbing $X$, and $\xi$ is induced by a Stein structure on $X$. Our technique uses an algorithm of Ozsv\'ath and Szab\'o  to determine the Heegaard-Floer homology of such $3$-manifolds. We discuss two important applications of this technique in contact topology. First, we show that it simplifies the calculation of the Ozsv\'ath-Stipsicz-Szab\'o obstruction  to admitting a planar open book. Then we define a numerical invariant of  contact manifolds that respects a partial ordering induced by Stein cobordisms. We do a sample calculation showing that the invariant can get infinitely many distinct values.

\end{abstract}

\maketitle

\section{Introduction}

\vspace{0.3cm}

The last decade was the scene of many achievements in contact topology in dimension three. In year 2000, in his seminal work \cite{G}, Giroux established a one to one correspondence between contact structures and open book decompositions of closed oriented $3$-manifolds. This allowed Ozsv\'ath and Szab\'o  to find a Heegaard-Floer homology class  that reflects certain properties of a given contact structure, \cite{OS2}. In another direction, based on Giroux's work, Ozbagci and Etnyre \cite{EO} defined an invariant, the support genus, which is the minimal page genus of an open book decomposition compatible with a fixed contact structure. Previously, Etnyre \cite{E1} had found out that being supported by a genus zero open book puts some restrictions on intersection forms of symplectic fillings of a contact structure. His result was later improved by Ozsv\'ath, Stipsicz and Szab\'o  who showed that the image of the Ozsv\'ath-Szab\'o  contact invariant  in the reduced version of Heegaard-Floer homology is actually an obstruction to be supported by a planar open book. More precisely,  they proved the following.

\vspace{0.2 cm}

\begin{theo}\label{goss}(Theorem 1.2 in \cite{OSS})
 Suppose that the contact structure $\xi$ on $Y$ is compatible with a planar
open book decomposition. Then its contact invariant $c^+(\xi)\in HF^+(-Y )$ is contained in
$U^d \cdot HF^+(-Y )$ for all $d \in \mathbb{N}$.
\end{theo} 

\vspace{0.2 cm}

 In spite of having useful corollaries, this theorem may not be easy to apply all the time because it involves calculation of the group $HF^+$ and identification of the contact invariant in this group. The former problem can be solved if we restrict our attention to a certain class of  manifolds.  In \cite{OS1}, Ozsv\'ath and Szab\'o  gave a purely combinatorial description of Heegaard-Floer homology groups $HF^+$ of some $3$--manifolds which are given as the boundary of certain plumbings of disk bundles over sphere. The  present work is about pinning down the contact element within this combinatorial object.

\vspace{0.2 cm}

To state our main results, we shall assume that $G$ is a negative definite plumbing tree with at most one bad vertex. Let $X(G)$ and $Y(G)$ be the $4$-- and $3$--manifolds determined by the plumbing diagram respectively. Denote the set of all characteristic co-vectors of the lattice $H^2(X(G),\mathbb{Z})$ by $\mathrm{Char}(G)$. We form the group $\mathbb{K}^+(G)=(\mathbb{Z}^{n\geq 0}\times \mathrm{Char}(G))/\sim$ where the relation $\sim$ is to be described in Section \ref{algorithm}. Recall that the Heegaard-Floer homology group $HF^+$ of any $3$-manifold is equipped with an endomorphism $U$. In \cite{OS1} (see also Section \ref{algorithm} below), Ozsv\'ath and Szab\'o established the following isomorphism.

\begin{equation}\label{idenfloe}
 \mathrm{Hom} \left (\frac{\mathbb{K}^+(G)}{\mathbb{Z}^{>0}\times\mathrm{Char}(G)},\mathbb{F}\right )\simeq \mathrm{Ker}(U)\subset HF^+(-Y(G))
\end{equation}

Recall that if $\xi$ is a contact structure, its Ozsv\'ath-Szab\'o  contact invariant $c^+(\xi)$ is a homogeneous element in $\mathrm{Ker}(U)\subset HF^+(-Y(G))$. It is also known that $c^+(\xi)$ is non-zero if $\xi$ is induced by a Stein filling. The following proposition pins down the image of contact invariant under the above isomorphism. 

\vspace{0.5cm}

\begin{prop}\label{repr}
Let $J$ be a Stein structure on $X(G)$ and $\xi$ be the induced contact structure on $Y(G)$.  Under the identification described in Equation \ref{idenfloe}, the contact invariant $c^+(\xi)$ is represented by the dual of the first Chern class $c_1(J)\in H^2(X,\mathbb{Z}) $ .
\end{prop}

\vspace{0.5cm}

\begin{rem}\label{rem:repr}
This proposition can be generalized in several different directions. First, we may allow the graph $G$ to have two bad vertices. In this case, the group  on the left hand side of Equation \ref{idenfloe} gives only even degree elements in $\mathrm{Ker}(U)\subset HF^+(-Y(G))$.  Second, the graph $G$ which has at most one bad vertex could be semi-definite implying that $b_1(Y)=1$, and we use the generalization of the Ozsv\'ath-Szab\'o  algorithm given in \cite{R2}. Finally, keeping $G$ negative definite, we may require $J$ to be an $\omega-$tame almost complex structure on $X(G)$ for some symplectic structure $\omega$ which restricts positively on the set of complex tangencies of $Y(G)$ (i.e. $(X(G),\omega)$ forms a weak filling rather than a Stein filling for the corresponding contact structure on the boundary).
\end{rem}

\vspace{0.5cm}

 When combined with Theorem \ref{goss}, Proposition \ref{repr} allows us to determine whether or not certain contact structures admit planar open books. Recall that the correction term for any $\mathrm{spin}^c$ structure $\mathfrak{t}$ of a rational homology $3$-sphere $Y$ is the minimal degree  of any non-torsion class in $HF^+(Y,\mathfrak{t})$ coming from $HF^\infty(Y,\mathfrak{t})$.

\vspace{0.5cm}

\begin{theo}\label{rest}
Let $J$ be a Stein structure on $X(G)$ and $\xi$ be the induced contact structure on $Y(G)$. Denote the correction term of the induced $\mathrm{spin}^c$ structure $\mathfrak{t}$ on $Y(G)$ by $d$. Also, let $d_3(\xi)$ be the $3$-dimensional invariant of the contact structure $\xi$. Suppose that we have either $d_3(\xi)\neq -d-1/2$ or $\mathrm{rank}(HF^+_d(-Y(G),\mathfrak{t}))>1$ then $\xi$ can not be supported by a planar open book.  
\end{theo}

\vspace{0.5cm}

%\newpage

Note that checking the conditions stated in this theorem is simply a combinatorial matter, \cite{OSS} (see also Section \ref{algorithm} below). Corollary 1.7 of \cite{OSS}, which holds for arbitrary rational homology 3-spheres, implies the above statement when $d_3\neq -d(\xi)-1/2$. This could be taken as an evidence to conjecture that Theorem \ref{rest} also holds for every rational homology $3$-sphere.

\vspace{0.3cm}

\begin{rem}

There is another version of Ozsv\'ath-Szab\'o  contact invariant $c(\xi)$ which lives in $\widehat{HF}(-Y)$ and is related to $c^+(\xi)$ by $\iota(c(\xi))=c^+(\xi)$ where $\iota$ is the natural map $\iota:\widehat{HF}(-Y)\to HF^+(-Y)$. The invariant $c(\xi)$ can be calculated combinatorially as shown in \cite{P2} and \cite{BP}. However, for the present applications the usage of the $c^+$ is essential. 

\end{rem}

\vspace{0.3cm}

The techniques of this paper can also be used to study a natural partial ordering on contact $3-$manifolds up to some equivalence. Following \cite{EH} and \cite{Ga}, we write $(Y_1,\xi_1)\preceq (Y_2,\xi_2)$ if there is a Stein cobordism from  $(Y_1,\xi_1)$ to $(Y_2,\xi_2)$. Moreover, we  write $(Y_1,\xi_1)\sim(Y_2,\xi_2)$ if these contact manifolds satisfy $(Y_1,\xi_1)\preceq (Y_2,\xi_2)$ and conversely $(Y_2,\xi_2)\preceq (Y_1,\xi_1)$. Clearly, $\sim$ defines an equivalence relation on the set of contact manifolds and $\preceq$ is a partial ordering on the equivalence classes. One can define a numerical invariant of contact manifolds that respects this partial ordering. Namely, if we let

$$\sigma(Y,\xi)=-\max \left \{d:c^+(\xi)\in U^d \cdot HF^+(-Y ) \right \}$$

\noindent the naturality properties of the Ozv\'ath-Szab\'o  contact invariant (c.f. Section \ref{homology} below) imply that we have $\sigma(Y_1, \xi_1)\leq \sigma(Y_2, \xi_2)$ whenever $(Y_1,\xi_1)\preceq (Y_2,\xi_2)$. Note that $\sigma$ invariant can be infinite. In fact,  $\sigma(Y,\xi)=-\infty$ if $(Y,\xi)$ admits a planar open book   by Theorem \ref{goss}.  Clearly,  if two contact manifolds have different  $\sigma$-invariants, they lie in different equivalence classes. The following theorem tells that there are infinitely many such equivalence classes.

\vspace{0.2cm}

\begin{theo}\label{nesigma}
Any negative integer can be realized as the $\sigma$ invariant of a contact manifold.
\end{theo}

In fact, we are going to obtain some contact manifolds with distinct $\sigma$ invariants by doing Legendrian surgery on certain stabilizations of some torus knots. See Theorem \ref{sigma} below. After completing the first draft of this paper, the author found an explicit formula for the $\sigma$ invariant of a contact manifold that is obtained by Legendrian surgery from $3$-sphere if the knot has $L$--space surgery \cite{K}. The formula depends only the Alexander polynomial, Thurston--Bennequin number and the rotation number of the surgery knot and it generalizes Theorem \ref{sigma}. The technique, however, is quite different than the one used here. 

\begin{rem}
Recently Latschev and Wendl defined an analogous invariant of contact manifolds, which they call \emph{algebraic torsion}, in arbitrary odd dimension within the context of Symplectic Field Theory, \cite{LW}. In dimension $3$, both invariants provide obstructions to exact symplectic cobordisms, so one may wonder if these two are somehow related. So far, we can not see an obvious relation, because Theorem 1.1 in \cite{LW} says that contact manifolds with algebraic torsion are not strongly fillable whereas our examples with finite $\sigma$ invariant are all Stein fillable. 
\end{rem}

\vspace{0.5cm}

This paper is organized as follows. In Section \ref{homology}, basic properties of Heegaard-Floer homology and  contact invariant are briefly reviewed. Section \ref{algorithm} is devoted to the algorithm of Ozsv\'ath and Szab\'o  to determine the generators of Heegaard-Floer homology of 3-manifolds given by plumbing diagrams. Remarks given at the end of the section allow us to find relations easily by combinatorial means. We prove Proposition \ref{repr} and Theorem \ref{rest} in Section \ref{theoremo}. Some examples are given in Section \ref{sample}. We discuss the planar obstruction in Section \ref{exampleo}. Theorem \ref{nesigma} is proved in Section \ref{sec:sigma}

\vspace{0.5 cm}

I would like thank my advisor Selman Akbulut for his patience and constant encouragement. I am grateful to Tolga Etg\"{u}, Matt Hedden and Yank\i \ Lekili for helpful conversations. A special thanks goes to Burak Ozbagci for his careful revision of the first draft of this paper. This work is supported by a Simons postdoctoral fellowship.

\section{Heegaard-Floer homology and contact invariant}\label{homology}

\vspace{0.3 cm}

Let $Y$ be a closed oriented $3$-manifold and $\mathfrak{t}$ be a $\mathrm{spin}^c$ structure on $Y$. In \cite{OS4} and \cite{OS5}, Ozsv\'ath and Szab\'o  define four versions of \emph{Heegaard-Floer homology} groups $\widehat{HF}(Y,\mathfrak{t})$,  $HF^+(Y,\mathfrak{t})$, $HF^-(Y,\mathfrak{t})$,and $HF^\infty(Y,\mathfrak{t})$. These groups are all smooth invariants of $(Y,\mathfrak{t})$.  When $Y$ is a rational homology sphere, they admit absolute $\mathbb{Q}$-gradings. The groups $HF^+$, $HF^-$, and $HF^\infty$ are also $\mathbb{Z}[U]$ modules where multiplication by $U$ decreases degree by $2$. Any $\mathrm{spin}^c$ cobordism $(X,\mathfrak{s})$ between $(Y_1,\mathfrak{t}_1)$ and $(Y_2,\mathfrak{t}_2)$ induces a homomorphism well defined up to sign

$$F_{X,\mathfrak{s}}^\circ:HF^\circ(Y_1,\mathfrak{t}_1)\to HF^\circ(Y_2,\mathfrak{t}_2)$$

\vspace{0.2cm}

Here $HF^\circ$ stands for any one of $\widehat{HF}$, $HF^+$, $HF^-$, or $HF^\infty$.  We work with $\mathbb{F}=\mathbb{Z}/2\mathbb{Z}$ coefficients in order to avoid sign ambiguities. 
Also, we drop the $\mathrm{spin}^c$ structure from the notation when we direct sum over all $\mathrm{spin}^c$ structures.

\vspace{0.5cm}

Given any contact structure $\xi$ on $Y$, Ozsv\'ath and Szab\'o  associate an element $c(\xi)\in \widehat{HF}(-Y)$ which is an invariant of isotopy class of $\xi$  \cite{OS2}. In this paper we are interested in the image $c^+(\xi)$ of $c(\xi)$ in $HF^+(Y)$ under the natural map. We list some of the properties of this element below.

\vspace{0.2cm}

\begin{enumerate}
\item $c^+(\xi)$ lies in the summand $HF^+(-Y,\mathfrak{t})$ where $\mathfrak{t}$ is the $\mathrm{spin}^c$ structure induced by $\xi$.
\item $c^+(\xi)=0$ if $\xi$ is overtwisted.
\item $c^+(\xi)\neq 0$ if $\xi$ is Stein fillable.
\item $c^+(\xi)\in \mathrm{Ker}(U)$.
\item $c^+(\xi)$ is homogeneous. When $Y$ is a rational homology sphere, it has  degree $-d_3(\xi)-1/2$, where $d_3(\xi)$ is the $3$-dimensional invariant of $\xi$.
\item $c^+(\xi)$ is natural under Stein cobordisms: If $W$ is a compact Stein manifold, $\partial W= Y'\cup -Y$, and $\xi '$ and $\xi$ are the induced contact structures, we can regard $W$ as a cobordism from $-Y'$ to $-Y$ and the induced map satisfies $F^+_{W}(c^+(\xi '))=c^+(\xi)$. 
\end{enumerate}

\vspace{0.5cm}

The contact invariant $c^+(\xi)$ is studied by Plamenevskaya in \cite{P}. The following result is to be used later in this paper when we prove our main theorem. We state it in a slightly more general form than in \cite{P} but Plamanevskaya's proof is valid for our case as well.

\vspace{0.2 cm}

%\newpage

\begin{theo}(Theorem 4 in \cite{P})\label{Plame}
\noindent Let X be a smooth compact $4$-manifold with boundary $Y = \partial X$.
Let $J$ be a Stein structure on $X$ that induces a $\mathrm{spin}^c$ structure $\mathfrak{s}_1$ on $X$
and contact structure $\xi_1$ on $Y$. Let $\mathfrak{s}_2$ be another $\mathrm{spin}^c$ structure on $X$ that does not necessarily come from a Stein structure.  Suppose that $\mathfrak{s}_1|_Y = \mathfrak{s_2}|_Y$ , but the $\mathrm{spin}^c$ structures $\mathfrak{s}_1$, $\mathfrak{s}_2$ are
not isomorphic. We puncture $X$ and regard it as a cobordism
from $Y$ to $S^3$. Then

\vspace{0.2 cm}

\begin{enumerate}
	\item $F^+_{X,\mathfrak{s_2}}(c^+(\xi_1))=0$
	\item $F^+_{X,\mathfrak{s_1}}(c^+(\xi_1))$ is a generator of $HF_0^+(S^3)$. 
\end{enumerate}
\end{theo}

\vspace{0.2cm}

\noindent Note that in this theorem if the $\mathrm{spin}^c$ structures $\mathfrak{s}_1|_Y$ and  $\mathfrak{s_2}|_Y$ are not the same then conclusion (1) follows trivially.

\vspace{0.2 cm}

\begin{rem}
Theorem \ref{Plame} was later generalized by Ghiggini in \cite{Ghi} where he requires $J$ to be only an $\omega$-tame almost complex structure for some symplectic structure $\omega$ on $X(G)$ that gives a strong filling for the boundary contact structure. In this paper, we work with rational homology spheres. For these manifolds, any weak filling can be perturbed into a strong filling, \cite{OO}.
\end{rem}

\section{The Algorithm}\label{algorithm}

\vspace{0.3cm}

In this section, we review Ozsv\'ath--Szab\'o's combinatorial description of Heegaard Floer homology of plumbed $3$--manifolds given in \cite{OS1} to set our notation.  Proof of our main theorem heavily relies on the understanding the algebraic structure of their combinatorial description. Particularly one should understand the $U$--action in this combinatorial object. We shall  describe this action in Equation \ref{Uaction}. Strictly speaking, Ozsv\'ath and Szab\'o's algorithm determines only the part of the Heegaard-Floer homology group that lies in the kernel of $U$ map. In order to determine the full group, one should find all the minimal relations. Although these relations can be found in some special cases,  no general technique is known to find all of these relations. Towards the end of the section we discuss a systematic method  to find some (not necessarily minimal) relations. These relations will turn out to be minimal in the cases of interest (Example \ref{exam:star}).   

\vspace{0.2cm}

Let $G$ be a weighted graph. For every vertex $v$ of $G$, let $m(v)$ and $d(v)$ denote the weight of $v$ and the number of edges connected to $v$ respectively. A vertex $v$ is said to be a \emph{bad vertex} if $m(v)+d(v)>0$. Enumerating all vertices of $G$, one can form the \emph{intersection matrix} whose $i$th diagonal entry is $m(v_i)$ and $i-j$th entry is $1$ if there is an edge between $v_i$ and $v_j$, and is $0$ otherwise. Throughout, we assume that $G$ satisfies the following conditions.

\vspace{0.2cm}

\begin{enumerate}
\item $G$ is a connected tree.
\item The intersection matrix of $G$ is negative definite.
\item G has at most one bad vertex.
\end{enumerate}

\vspace{0.2cm}

There is a $4$-manifold $X(G)$ obtained by plumbing together disk bundles $D_i,\;i=1\cdots|G|$ over sphere where $D_i$ is plumbed to $D_j$ whenever there is an edge connecting $v_i$ to $v_j$. Let $Y(G)$ be the boundary of $X(G)$. In \cite{OS1}, Ozsv\'ath and Szab\'o  give a purely combinatorial description of Heegaard-Floer homology group $HF^+(-Y(G))$.  From now on, we identify $\mathrm{spin}^c$ structures on $4$--manifolds with their first Chern classes. Since all our $4$--manifolds are simply connected and have non-empty boundary, this does not cause any ambiguity. However, we should be careful in the $3$--manifold level when  $2$--torsion exists in the first homology. We will deal with such an example in Section  \ref{exampleo}(see Remark \ref{ambig}).

\vspace{0.5cm}

The second homology $H^2(X(G),\mathbb{Z})$ is a free module generated by vertices of $G$. Let $\mathrm{Char}(G)$ be the set all characteristic(co)vectors of this module, i.e. every element $K$ of $\mathrm{Char}(G)$ satisfies $\langle K,v\rangle=m(v) \; (\text{Mod} \; 2)$ for every vertex $v$. Let $\mathcal{T}^+$ be the graded algebra $\mathbb{F}[U,U^{-1}]/U\mathbb{F}[U]$ where the formal variable $U$ has degree $-2$. Form the set $\mathbb{H}^+(G)\subset\mathrm{Hom}(\text{Char}(G),\mathcal{T}^+)$ where any element $\phi$ of $\mathbb{H}^+(G)$ satisfies the following property; If $K$ is a characteristic vector, $v$ is a vertex, and $n$ is an integer such that
$$\langle K,v \rangle + m(v)=2n,$$
we have
$$U^{m+n}\phi(K+2\mathrm{PD}(v))=U^m\phi(K) \;\text{if}\;n>0,$$
or
$$U^{m}\phi(K+2\mathrm{PD}(v))=U^{m-n}\phi(K) \;\text{if}\;n<0.$$

\vspace{0.2 cm}

The set of $\mathrm{spin}^c$ structures on $Y(G)$ gives rise to a natural splitting for $\mathbb{H}^+(G)$. For, if $\mathfrak{t}$ is a $\mathrm{spin}^c$ structure on $Y(G)$, one can consider the subset $\mathrm{Char}_\mathfrak{t}(Y(G))$ consisting of those characteristic vectors whose restriction on $Y(G)$ are $\mathfrak{t}$. The set $\mathbb{H}^+(G,\mathfrak{t})$ is the set of all maps in $\mathbb{H}^+(G)$ with support $\mathrm{Char}_\mathrm{t}$. It is easy to see that  $ \mathbb{H}^+(G)=\bigoplus_\mathfrak{t}\mathbb{H}^+(G,\mathfrak{t})$ .

The group $\mathbb{H}^+(G)$ is graded in the following way. An element $\phi \in \mathbb{H}^+(G)$ is said to be homogeneous of degree $d$ if for every characteristic vector $K$ with $\phi (K) \neq 0$, $\phi (K) \in \mathcal{T}^+$ is a homogeneous element with

$$\mathrm{deg}(\phi(K))-\frac{K^2+|G|}{4}=d.$$ 

\vspace{0.1cm}

We are ready to describe the isomorphism relating $\mathbb{H}^+(G)$ to the Heegaard-Floer homology of $Y(G)$. Fix a $\mathrm{spin^c}$ structure $\mathfrak{t}$ on $-Y(G)$. Let $K$ be a characteristic vector on $\mathrm{Char}_\mathfrak{t}(G)$. Puncture $X(G)$ and regard it as a cobordism form $-Y(G)$ to $S^3$. It is known that $X(G)$ and $K$ induce a homomorphism

$$F_{X(G),K}:HF^+(-Y(G),\mathfrak{t})\to HF^{+}(S^3) \simeq \mathcal{T}^+.$$

\vspace{0.1cm}

Now the map $T^+:HF^+(-Y(G),\mathfrak{t})\to \mathbb{H}^+(G,\mathfrak{t})$ is defined by the rule $T^+(\xi)(K)=F_{X(G),K}(\xi)$

\begin{theo}(Theorem 2.1 in \cite{OS1})
$T^+$ is a $U$-equivariant isomorphism preserving the absolute $\mathbb{Q}$-grading.
\end{theo}

To simplify the calculations, we work with the dual of $\mathbb{H}^+(G)$. Let $\mathbb{K}^+$ be the quotient set $\mathbb{Z}^{\geq 0}\times \mathrm{Char}(G)/\sim$, where the equivalence relation $\sim$ is defined as follows. Denote a typical element of $\mathbb{Z}^{\geq 0}\times \mathrm{Char}(G)$ by $U^m \otimes K$(We drop $U^m\otimes\;$ from our notation if $m=0$).  Let $v$ be a vertex and $n$ be an integer such that
$$ 2n= \langle K,v \rangle + m(v)$$
then we have 
$$U^{m+n} \otimes (K+2\mathrm{PD}(v)) \sim U^{m}\otimes K \;\mathrm{if}\;\mathrm{ n\geq 0}$$
or
$$U^{m}\otimes(K+2\mathrm{PD}(v) \sim U^{m-n}\otimes K \;\mathrm{if}\;\mathrm{ n< 0}.$$

\vspace{0.2cm}

Define a pairing $\mathbb{K}^+(G)\times \mathbb{H}^+(G)\to \mathbb{Z}$ by $(\phi,U^m\otimes K) \to (U^m\phi(K))_0$ where $()_0$ denotes the projection to the degree 0 subspace of $\mathcal{T}^+$. It is possible to show that this pairing is well defined and non-degenerate and hence it defines an isomorphism between $\mathbb{H}^+(G)$ and $\mathrm{Hom}(\mathbb{K}^+(G),\mathbb{Z})$. Using the duality map and isomorphism $T^+$ one can identify $\ker U^{n+1} \subset HF^+(-Y(G))$ as a quotient of $\mathbb{K}^+(G)$ for every $n\geq 0$.

\begin{lemm}\label{finmod}(Lemma 2.3 in \cite{OS1}) Let $B_n$ denote the set of characteristic vectors $B_n=\{K\in\mathrm{Char(G)}:\forall v \in G, |\langle K,v\rangle | \leq -m(v)+2n\}$. The quotient map induces a surjection from
$$\bigcup_{i=0}^n U^i\otimes B_{n-i}$$
onto the quotient space
$$\frac{\mathbb{K}^+(G)}{\mathbb{Z}^{>n}\times \mathrm{Char}(G)}.$$
In turn, we have an identification

\begin{equation}\label{modeliso}
\mathrm{Hom}\left ( \frac{\mathbb{K}^+(G)}{\mathbb{Z}^{>n}\times\mathrm{Char(G)}},\mathbb{F} \right)\simeq \ker U^{n+1}\subset \mathbb{H}^+(G).
\end{equation}
\end{lemm}

One should regard the above isomorphism as one between $\mathbb{F}[U]$ modules where the $U$ action on the left hand side of equation \ref{modeliso} is defined by the following relation.

\vspace{0.2 cm}

\begin{equation}\label{Uaction}
U.(U^p\otimes K)^{*}(U^r\otimes K')=  \left \{
\begin{array}{cc}
1 & \;\;\; \mathrm{if}\; U^p\otimes K \sim U^{r+1}\otimes K'\\
  &\\
0 & \;\;\; \mathrm{if}\; U^p\otimes K \not \sim U^{r+1}\otimes K'
\end{array}\right.
\end{equation}

\noindent Where $(U^p\otimes K)^{*}$ denotes the dual of $U^p\otimes K$.

\vspace{0.2 cm}

Lemma \ref{finmod} gives us a finite model for $\ker U^{n+1}$ for every $n\geq 0$. It is known that these groups stabilize to give $HF^+$. Therefore, one can understand $HF^+$ by studying the quotients $\mathbb{K}^+(G)/\mathbb{Z}^{\geq n}\times\mathrm{Char}(G)$ for all $n\geq 0$. The first quotient is well understood thanks to an algorithm of Ozsv\'ath and Szab\'o . Below, we describe the algorithm and discuss a possible extension.

\newpage

A characteristic vector $K$ is called an \emph{initial vector} if for every vertex $v$, we have

\begin{equation}
m(v)+2\leq \langle K,v \rangle \leq -m(v) \label{ini}
\end{equation}

\vspace{0.2 cm}

Start with an initial vector $K_0$. Form a sequence $(K_0,K_1,\cdots,K_n)$ of characteristic vectors as follows: $K_{i+1}$ is obtained from $K_i$ by adding $2\mathrm{PD}(v)$ where $v$ is a vertex with $\langle K_i,v\rangle=-m(v)$. The terminal vector $K_n$ satisfies one of the following.

\vspace{0.2 cm}

\begin{enumerate}
\item\label{pro1} $m(v)\leq \langle K_n,v \rangle \leq -m(v)-2$ for all $v$.
\item\label{pro2} $\langle K_n,v \rangle > -m(v)$ for some $v$.
\end{enumerate}

\vspace{0.2 cm} 

The sequence $(K_0,K_1,\cdots,K_n)$ is called a \emph{full path}, and characteristic vector ${K_n}$ is called the \emph{terminal vector} of the full path. We say that a full path is called \emph{good} if its terminal vector satisfies property (\ref{pro1}) above and it is \emph{bad} if the terminal vector satisfies (\ref{pro2}). We list some of the properties of full paths, the reader can consult \cite{OS1}(especially proposition 3.1 in \cite{OS1}) for proofs.

\vspace{0.2 cm}

\begin{itemize}
\item  Two characteristic vectors in $B_0$ are equivalent in $\mathbb{K}^+(G)$ if and only if there is a full path containing both of them where the set $B_0$ is defined as in lemma \ref{finmod}.
\item If an initial vector $K_0$ has a good full path then any other full path starting with $K_0$ is good.
\item If $K_0$ and $K_0'$ are initial vectors having good full paths and $K_0\neq K_0'$ then  $K_0 \not \sim K_0'$ in $\mathbb{K}^+(G)$.
\item A terminal vector $K_n$ of a bad full path is equivalent to $U^m\otimes K'$ in $\mathbb{K}^+(G)$ for some $m>0$ and $K'\in \mathrm{Char}(G)$. A terminal vector of good full path can not be equivalent to such an element of $\mathbb{H}^+(G)$.
\end{itemize}

\vspace{0.2 cm}

Note that these properties allow us to find the generators of $\ker U$; They are simply the initial vectors having good full paths. In other words, we know the generators of the lowest grade subgroup of $HF^+(-Y(G))$. Recall from \cite{OS3} that the lowest degree $d(Y,\mathfrak{t})$ of non-torsion elements in $HF^+(Y,\mathfrak{t})$ is called the \emph{correction term} for a $\mathrm{spin}^c$ manifold $(Y,\mathfrak{t})$. The algorithm above provides us  an efficient method  to calculate the correction term $d(-Y(G),\mathfrak{t})$ for any $\mathrm{spin}^c$ structure $\mathfrak{t}$ (see Corollary 1.5 of \cite{OS1})

\vspace{-0.2cm}

\begin{equation}\label{correct}
d(-Y(G),\mathfrak{t})=\mathrm{min} -\frac{K^2+|G|}{4}
\end{equation}

\vspace{0.2cm}

\noindent where the minimum is taken over all characteristic vectors admitting good full paths which induce the $\mathrm{spin}^c$ structure $\mathfrak{t}$.

\vspace{0.5cm}

The whole group $\mathbb{H}^+(G)\simeq HF^+(-Y(G))$ is determined by the relations amongst the generators of $\mathrm{Ker}(U)$. Given two characteristic vectors $K_i$, $K_j$ admitting good full paths and inducing the same $\mathrm{spin}^c$ structure on $Y(G)$, a \emph{relation} between  $K_1$ and $K_2$ is a pair of integers $(n,m)$  satisfying $U^n\otimes K_1\sim U^m\otimes K_2$. If the non negative integers $(n,m)$ are minimal with that property, we call the corresponding relation \emph{minimal}. Here we describe a systematic method to find relations. Say $K$ is a characteristic  vector and $n$ is a positive integer. We want to understand the equivalence class in $\mathbb{K}^+(G)$ containing $U^n \otimes K$. We define three operations that do not change this equivalence class.

\vspace{0.3cm}

\textbf{(R1)} $U^n \otimes K' \to U^n \otimes K$ 
where $K'$ is obtained from $K$ by applying the algorithm to find full paths.

\textbf{(R2)} $U^n \otimes K \to U^{n-1} \otimes (K+2\mathrm{PD}(v))$ where $v$ is a vertex with $ \langle K,v \rangle + m(v) = -2$

\textbf{(R3)} $U^n \otimes K \to U^{n+1} \otimes (K+2\mathrm{PD}(v))$ where $v$ is a vertex with $ \langle K_n,v \rangle +m(v) = 2$

\vspace{0.3cm}

Now assume that $K$ is a characteristic initial vector which admits a good full path. In order to find particular representatives with small $U$-depths for the equivalence class containing $U^n\otimes K$ we apply \textbf{R1} then apply \textbf{R2} if possible else \textbf{R3}. Then we repeat the same procedure till it terminates at an element $U^r\otimes K'$. We call the vector part of this element as a \emph{root vector} (the exponent $m$ is determined by $n$ and degrees of $K$ and $K'$). A root vector is not unique, it depends upon  choices we made along the way; like  the choice of the vertex at which we apply \textbf{R2} or \textbf{R3} is applied. However, the set of root vectors is a finite set which can be found easily and it can be used to establish relations amongst the generators of $\mathrm{Ker}(U)$. This simple observation will be useful when we do our calculations.

\vspace{0.2 cm}

\begin{prop}
Let $K_1$ and $K_2$ be two characteristic initial vectors admitting good full paths. Suppose $n$ and $m$ are non-negative integers such that the root vector sets of $U^n\otimes K_1$ and $U^m\otimes K_2$ intersect non trivially. Then we have $U^n\otimes K_1 \sim U^m\otimes K_2$.
\end{prop}

\begin{proof}
Follows from the definitions.
\end{proof}

\vspace{0.2 cm}

\section{Main theorem}\label{theoremo}
\textbf{ Proof of Proposition \ref{repr}.} Let $\mathfrak{s}$ be the canonical $\mathrm{spin}^c$ structure  and $\mathfrak{s}'$ be any other $\mathrm{spin}^c$ structure on $X(G)$. Note that $c_1(\mathfrak{s})\not \sim c_1(\mathfrak{s}')$. 
Recall that the isomorphism $\displaystyle \mathrm{Ker}(U)\simeq\mathrm{Hom}\left ( \frac{\mathbb{K}^+(G)}{\mathbb{Z}^{>0}\times\mathrm{Char}(G)},\mathbb{F}\right )$ is given by means of the pairing 

$$P:\displaystyle \mathrm{Ker}(U)\times\mathrm{Hom}\left ( \frac{\mathbb{K}^+(G)}{\mathbb{Z}^{>0}\times\mathrm{Char}(G)},\mathbb{F}\right )\to\mathbb{F}$$
$$P(a,L)=(F^+_{X(G),L}(a))_0$$

\noindent In view of this observation, it is enough to show the following two equations hold.

\begin{eqnarray}
(F^+_{X(G),\mathfrak{s}}(c(\xi)))_0&=&1\label{eq:rep1} \\
(F^+_{X(G),\mathfrak{s}'}(c(\xi))_0&=&0\label{eq:rep2}
\end{eqnarray}

\noindent These are simply the conclusions of Theorem \ref{Plame}.
\begin{flushright}
$\Box$
\end{flushright}

\textbf{Proof of Theorem \ref{rest}.}
Let $K=c_1(J)$.  By Theorem  \ref{goss} and Proposition \ref{repr}, it is enough to show that  $K^* \notin \mathrm{Im}(U^k)$ for some $k \in \mathbb{N}$. To do that we will use the the identification in Equation \ref{modeliso}, keeping in mind that the $U$ action is determined by  Equation \ref{Uaction}. Let $\{K_1,K_2,\cdots,K_r\}$ be the set of characteristic initial vectors  admitting good paths such that $\mathrm{deg}(K_i^*)\leq\mathrm{deg}(K^*)=-d_3(\xi)-1/2$ and $K_i|_{Y(G)}=\mathfrak{t}$ for all $i=1\cdots r$. Basic properties of the contact invariant imply that this set is not empty if one of the assumptions is satisfied. It is known that on any rational homology sphere and for any $\mathrm{spin}^c$ structure, the Heegaard-Floer homology decomposes as  $HF^+=\mathcal{T}^+\oplus HF_{\mathrm{red}}$ . This decomposition tells that in large even degrees the Heegaard-Floer homology is generated by a single element. So, one can find integers $n_0,n_1,\cdots,n_r$ such that
$$U^{n_0}\otimes K\sim U^{n_1}\otimes K_1\sim\cdots\sim U^{n_r}\otimes K_r.$$  
Moreover, by choosing these numbers large enough, we can guarantee that the dual of $U^{n_0}\otimes K$ is the unique generator of the degree $-d_3(\xi)-1/2+2n_0$ subspace of $HF^+(-Y)$. Then by Equation \ref{Uaction},

$$U^{n_0}(U^{n_0}\otimes K)^*= K^*+ (U^{n_1-n_0}\otimes K_1)^*+\cdots+ (U^{n_r-n_0}\otimes K_r)^*.$$

\noindent Therefore $K\notin\mathrm{Im}(U^{n_0})$.
\begin{flushright}
$\Box$
\end{flushright}

\vspace{-0.2cm}
\section{Examples}\label{sample}

In this section, we shall discuss two examples. These examples have no particular importance on their, own but they are simple enough  to give a clear explanation of the ideas used in this paper.

\begin{exam}
Let $G$ be the graph indicated in Figure \ref{fig:config2}. Index the vertices so that the central one comes first. Our aim is to find all the characteristic co--vectors in the intersection lattice of $X(G)$ which admit good full path. We will denote each $K\in H^2(X(G))$ as a row vector $[\langle K, v_1\rangle,\cdots,\langle K, v_4 \rangle]$, where $v_i$ is the homology class of the sphere corresponding to  the $i$th vertex for all $i=1,\cdots,4$. If $K$ is characteristic and satisfies Inequality \ref{ini} then $\langle K,v_i\rangle=0,\;\mathrm{or}\; 2$ for every $i$. So we need  to find out which of the possible $16$ co-vectors admit good full paths. To represent full paths, we indicate the index of the vertex whose twice Poincare dual is added to the characteristic vector.  The algorithm terminates at the very first step for $K_1=[0,0,0,0]$. For $K_2=[0,2,0,0]$, we have the following good full path $2,1,3,4,1,2$. By  symmetry, $K_3=[0,0,2,0]$ and $K_4=[0,0,0,2]$ also admit good full path. For $[2,0,0,0]$, the full path $1,2,3,4$ terminates at a bad vector. Also it is easy to show that if $\langle K,v_i\rangle =2$ for more than one $i$ values then $K$ admits a bad full path.  Therefore $K_1,\cdots, K_4$ are the only characteristic co--vectors admitting good full path.

\begin{figure}[h]
\begin{center}
	\includegraphics[width=0.15\textwidth]{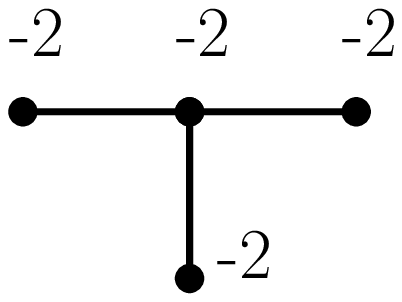}
		\caption{}
		\label{fig:config2}
	\end{center}
\end{figure}

\vspace{0.2cm}

Next we claim that each one of $K_1,\cdots,K_4$ restricts to a different $\mathrm{spin}^c$ structure $\mathfrak{t}_1,\cdots,\mathfrak{t}_4$ on the boundary. One way of seeing this is to apply the criterion mentioned in Remark \ref{ambig}.  Another way is the following: Recall that the set of $\mathrm{spin}^c$ structures on any $3$--manifold can be identified with its first homology. In this case the first homology of $Y(G)$ is given by $\mathbb{Z}^4/\mathrm{Im}I(G)$, where $I(G)$ is the intersection matrix. Observe that $\mathrm{det}(I(G))=4$, so $-Y(G)$ has  $4$  $\mathrm{spin}^c$ structures. Each one of these $\mathrm{spin}^c$ structures are torsion so the Heegaard-Floer homology of $-Y(G)$ is non-trivial in the corresponding component. Since we have exactly $4$ co--vectors contributing the Heegaard-Floer homology they must lie in different $\mathrm{spin}^c$ components. This shows $-Y(G)$ (and hence $Y(G)$) is an $L$--space (i.e. its Heegaard-Floer homology is the same as a Lens space).

\vspace{0.2cm} 

Let us calculate degree of each $K_i$. In the formula $\mathrm{deg}(K)=(K^2+|G|)/4$, the inverse of the intersection matrix should be used when squaring $K$. We see that $\mathrm{deg}(K_1)=1$ and $\mathrm{deg}(K_j)=0$ for  $j=2,3,4$. Since the isomorphism given in Equation \ref{modeliso} is given in terms of dual co--vectors, we should take the negative of the degrees when we think of $K_i$'s as elements of the Heegaard-Floer homology. As a result, $HF^+(-Y(G),\mathfrak{t}_1)=\mathcal{T}^+_{(-1)}$ and $HF^+(-Y(G),\mathfrak{t}_i)=\mathcal{T}^+_{(0)}$, for $i=2,3,4$.

\vspace{0.2cm}

Having calculated the Heegaard-Floer homology of the boundary, we now want to see how the Ozsv\'ath-Szab\'o invariant of a contact structure sits in this group. We equip $X(G)$ with the obvious Stein structure $J$: First make the attaching circles of handles corresponding to the vertices  Legendrian unknot with $tb=-1$, see Figure \ref{fig:neylegconfig2}. Since each handle is attached with framing $tb-1$, the unique Stein structure on the $4$--ball extends across these handles, \cite{E}. To identify the contact invariant, we need to determine the Chern class of $J$. The value of $c_1(J)(v_i)$ is given by the rotation number of the corresponding Legendrian unkot. In this case the rotation numbers are all $0$, so $c_1(J)=K_1$. Hence the Ozsv\'ath-Szab\'o invariant of the induced contact structure is the unique generator of $HF^+(-Y(G),\mathfrak{t}_1)$ in in degree $-1$. Note that the invariant is in the image of $U^k$ for every $k$, so we do not get any obstruction to planarity.

\begin{figure}[h]
\begin{center}
	\includegraphics[width=0.40\textwidth]{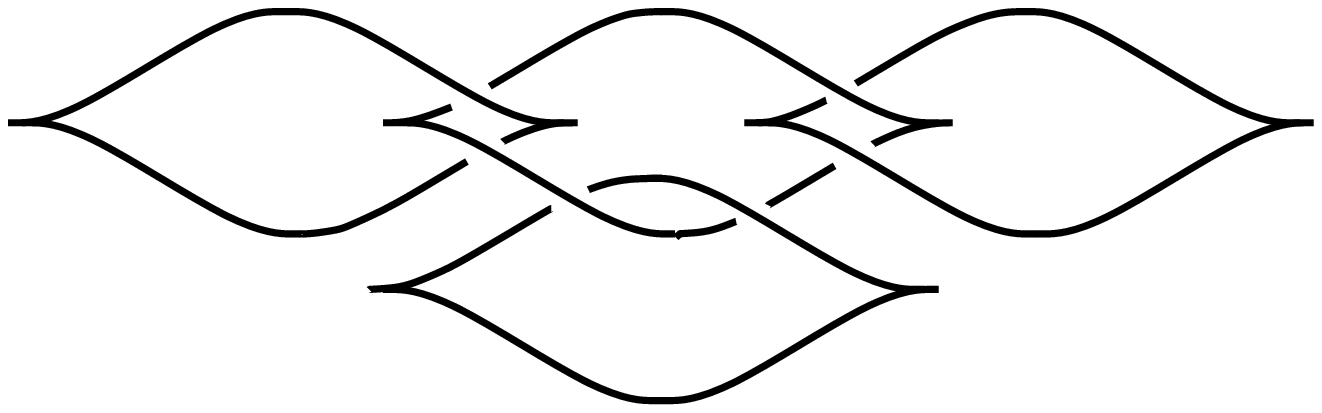}
	\caption{}
	\label{fig:neylegconfig2}
\end{center}
\end{figure}

\end{exam}

\begin{exam}
This example is a follow up of the calculation of the Heegaard Floer homology of the Brieskorn sphere $\Sigma(3,5,7)$ given in \cite{OS1}. This $3$--manifold is given by the plumbing graph $G$ which we indicate in  Figure \ref{fig:config4}. We order the vertices so that the central node comes first, the $-3$--sphere second, then the four vertices in the middle and finally, the six vertices on right. It is shown in \cite{OS1} that only the following characteristic co--vectors admit good full path.

\begin{center}
\begin{tabular}{ccc}
$K_1$&$=$&$(0,-1,0,0,0,0,0,0,0,0,0,0)$\\
$K_2$&$=$&$(0,\;1,0,0,0,0,0,0,0,0,0,0)$\\
$K_3$&$=$&$(0,1,0,0,0,0,0,0,0,0,0,-2)$\\
$K_4$&$=$&$(0,1,0,0,0,-2,0,0,0,0,0,0)$.
\end{tabular} 
\end{center}

\noindent We have $\mathrm{deg}(K_1)=\mathrm{deg}(K_2)=0$, and $\mathrm{deg}(K_3)=\mathrm{K_4}=2$. So the correction term for the unique $\mathrm{spin}^c$ structure is $-2$. Next we consider the Stein structures on $X(G)$. We make the each unknot Legendrian as before, but this time we should do a stabilization for the $-3$ framed unknot to get the framing $tb-1$. Depending on how we do the stabilization, we obtain two Stein structures $J_1$, $J_2$ whose Chern classes are given by $K_1$ and $K_2$. Let $\xi_1$ and $\xi_2$ respectively denote the induced contact structures on the boundary. Since $-\mathrm{deg}(K_i)$ is not minimal, neither contact structure is compatible with a planar open book by Theorem \ref{rest}. This also can be seen by using simple criteria found by Ozsv\'ath,Stipsicz and Szab\'o, see Theorem \ref{oss2} and \ref{oss3}.  

Finally, we would like to show why $c^+(\xi_i)$ is not in the image of $U$, for $i=1,2$. By Theorem \ref{repr} the contact invariant $c^+(\xi_i)$ is represented by the dual $K_i^*$. It was shown in \cite{OS1} that the minimal relations are given as follows
\begin{eqnarray*}
U\otimes K_3 \sim U\otimes K_4\\
U^2\otimes K_3 \sim U\otimes K_1 \sim U\otimes K_2
\end{eqnarray*}

\noindent Therefore $(U^2\otimes K_3)^*$ is the unique generator of degree $2$ and $U(U^2\otimes K_3)^*=(U\otimes K_3)^*+ K_1^* +K_2^*$, by Equation \ref{Uaction}. So neither $K_1^*$ nor $K_2^*$ is in the image of $U$.

\begin{figure}[h]
\begin{center}
	\includegraphics[width=0.30\textwidth]{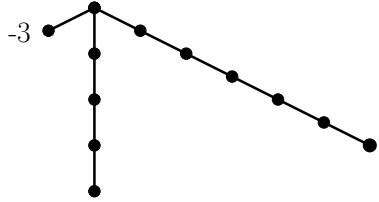}
	\caption{Plumbing graph for $\Sigma(3,5,7)$. Unlabeled vertices have weight $-2$}
	\label{fig:config4}
\end{center}
\end{figure}

\end{exam}
\section{Planar Obstruction}\label{exampleo}   

\vspace{0.2cm}

In this section, we shall illustrate an application Theorem \ref{rest} and show that certain Stein fillable contact structures do not admit planar open books. Obstructions to being supported by planar open books were known to exist before. Some of these obstructions can be checked by using simple criteria. The importance of our examples is that no other simple criterion is sufficient to prove their non-planarity. Before discussing our examples we shall give a brief exposition on what is known about obstruction to planarity.

\vspace{0.2cm}

The first known obstruction to planarity  was found by Etnyre. It puts some restrictions on intersection forms of symplectic fillings of planar open books.

\vspace{0.2cm}

\begin{theo}\label{etn1}(Theorem 4.1 in \cite{E1})
If $X$ is a symplectic filling of a contact $3$-manifold $(Y, \xi)$
which is compatible with a planar open book decomposition then $b_+ ^2 (X) = b_0 ^2
(X) = 0$, the
boundary of $X$ is connected and the intersection
form $Q_X$ embeds into a diagonalizable matrix over integers.
\end{theo}
\noindent Ozsv\'ath Szab\'o  and Stipsicz found another obstruction in \cite{OSS}. Their obstruction is a consequence of Theorem \ref{goss} above though its statement has no reference to Floer homology.
\begin{theo}\label{oss2}(Corollary 1.5 of \cite{OSS})
Suppose that the contact $3$-manifold $(Y, \xi)$ with $c_1(s(\xi)) = 0$ admits a
Stein filling ($X, J)$ such that $c_1(X, J) \neq 0$. Then $\xi$ is not supported by a planar open
book decomposition.
\end{theo}
\noindent Yet another criterion is stated in \cite{OSS}. It partially implies Theorem \ref{rest} above.
\begin{theo}\label{oss3}(Corollary 1.7 of \cite{OSS})
Suppose that $Y$ is a rational homology $3$-sphere. The number of homotopy
classes of $2$-plane fields which admit contact structures which are both symplectically
fillable and compatible with planar open book decompositions is bounded above
by the number of elements in $H_1(Y ; \mathbb{Z})$. More precisely, each $\mathrm{spin}^c$ structure $\mathfrak{s}$ is represented
by at most one such $2-$plane field, and moreover, the Hopf invariant of the
corresponding $2-$plane field must coincide with the "correction term" $d(-Y, s)$.
\end{theo}
Below, we give examples of  non-planar Stein fillable contact structures on a Seifert fibered space. Non-planarity of some of our examples do not follow from Theorem \ref{etn1}, Theorem \ref{oss2} or  Theorem \ref{oss3}. 
\begin{exam}\label{exam:star}
Consider the star shaped plumbing graph consisting of eight vertices where the central vertex has weight $-4$, a neighboring vertex has weight $-3$ and all the others are of weight $-2$ (See figure \ref{fig:config1}). The boundary $3$-manifold $Y$ is the Seifert fibered space $\displaystyle M(-4,\underbrace{\frac{1}{2},\cdots,\frac{1}{2}}_6,\frac{1}{3})$. The reason why we have so many self intersection  $-2$ spheres is that we want to avoid $L$-spaces where Theorem \ref{goss} does not provide an obstruction to admitting a planar open book. For the topological characterization of $L$-spaces among Seifert fibered spaces see \cite{LS}. It can be shown that the corresponding intersection form is negative definite,and  has determinant $128$. Moreover it can be embedded into a symmetric matrix which is diagonalizable over integers.  To see this, index the vertices so that the central one comes first and the weight $-3$ vertex is the last. Let $e_1,e_2,\cdots,e_{11}$ be a basis for $\mathbb{R}^{11}$ such that $e_i\cdot e_i=-1$ for all $i=1,2,\cdots,11$. The embedding is defined by the following set of equations

\begin{eqnarray}
\nonumber v_1&\to& -e_1-e_2-e_3-e_4\\
\nonumber v_2&\to& e_2-e_7\\
\nonumber v_3&\to& e_2+e_7\\
\nonumber v_4&\to& e_3-e_8\\
\nonumber v_5&\to& e_3+e_8\\
\nonumber v_6&\to& e_4-e_9\\
\nonumber v_7&\to& e_4+e_9\\
\nonumber v_8&\to& e_1+e_{10}+e_{11}
\end{eqnarray}

\vspace{0.3cm}

\begin{figure}[h]
\begin{center}
	\includegraphics[width=0.20\textwidth]{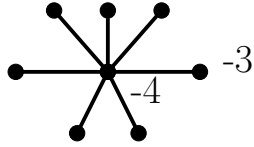}
	\caption{Plumbing description of Y. All unmarked vertices have weight -2.}
	\label{fig:config1}	
\end{center}
\end{figure}

First, we calculate $HF^+(-Y,\mathfrak{t})$ for every $\mathrm{spin}^c$ structure $\mathfrak{t}$. For similar calculations, see \cite{D}, \cite{R1}, \cite{R2}, and Section 3.2 of \cite{OS1}. As before, we  write any characteristic vector $K$ in the form $K=[ \langle K,v_1 \rangle,\cdots,\langle K,v_8 \rangle]$. There are $768$ characteristic vectors satisfying equality \ref{ini}, and $138$ of them have good full paths. When we distribute these to $\mathrm{spin}^c$ structures of $Y$, we see that for $10$ $\mathrm{spin}^c$ structures $\mathrm{Ker}(U)$ has rank $2$, and the rank is $1$ for the rest. Table \ref{fig:table1} shows $HF^+$ for these $10$ $\mathrm{spin}^c$ structures. 

\vspace{0.3cm}

\begin{rem}\label{ambig}
As pointed out in \cite{OS1}, the set of $\mathrm{spin}^c$ structures on $Y$ can be identified with $2H^2(X(G),\partial X(G))$ orbits in $\mathrm{Char}(G)$. Therefore two characteristic vectors $K_1$, $K_2$, induce the same $\mathrm{spin}^c$ structure on the boundary, if and only if all the entries of $(1/2) I(G)^{-1}(K_1-K_2)$ are integer where $I(G)$ is the intersection matrix.
\end{rem}

\vspace{0.3cm}

Next, we consider the obvious Stein structures that arise from the handlebody diagram associated to $G$. Following Eliashberg, we isotope the attaching circles of $2$-handles  into Legendrian position so that their framing become one less than the Thurston-Bennequin framing. For $-2$ framed unknots, there is unique way to do that. For the other unknots which correspond to $v_1$ and $v_8$ take Legendrian isotopes with rotation numbers $i$ and $j$ respectively where $i=-2,0,2$, and $j=-1,1$. Call the resulting Stein structure as $J_{i,j}$ and the induced contact structure by $\xi_{i,j}$, see figure \ref{fig:legconfig} for a picture of $J_{2,-1}$. Note that the first Chern class of $J_{i,j}$ is given by the characteristic vector $K_{i,j}=[i,0,0,0,0,0,0,j]$. It is easy to verify that $\displaystyle d_3(\xi_{i,j})+1/2=(K_{i,j}^2+|G|)/4=\mathrm{degree}(K_{i,j})$. According to Theorem \ref{rest},  the contact structures $\xi_{\pm 2, \pm 1}$ do not admit planar open books. By the algorithm given in \cite{EO2} these contact structures do admit genus one open books, so their support genera are all one. One can not use Theorem \ref{oss2}  directly to get this conclusion because the Chern classes of the corresponding $\mathrm{spin}^c$ structures are all of order $4$. Though Theorem \ref{oss3} also implies our conclusion for $\xi_{2,1}$ and $\xi_{-2,-1}$, it doesn't apply to $\xi_{2,-1}$ or $\xi_{-2,1}$. So the latter two are the contact structures we promised at the beginning of the example.

\begin{figure}[h]
\begin{center}
	\includegraphics[width=.40\textwidth]{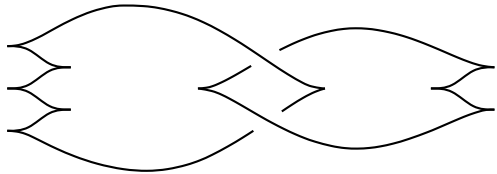}
	\caption{Legendrian handlebody diagram giving $J_{2,-1}$. The curve on left corresponds $v_1$ and the other represents $v_8$. They are both oriented counter clockwise. We omit the other unknots linking to $v_1$ in order not to complicate the picture.}
	\label{fig:legconfig}
	\end{center}
\end{figure}
 
\end{exam}

\begin{rem}
The main result of \cite{EO} implies that the support genera of plumbings with at most two bad vertices are at most one. On the other hand the algorithm of Ozsv\'ath and Szab\'o  does not work if the number of bad vertices is greater than two. Therefore, the techniques used in this paper do not seem to be sufficient to find an example of a contact structure with support genus strictly greater than one. We are planning to turn this problem in a future project using a different approach.
\end{rem}

\begin{table}[p]
\begin{center}

\begin{tabular}{|c|c|c|c|c|}
\hline
$\mathrm{Spin}^c$&Characteristic Vectors & Degree & Relation &  $ HF^+(-Y)$\\
\hline
\multirow{2}{*}{1}&$[2,0,0,0,0,0,0,-1]$&$7/8$&\multirow{2}{*}{$U\otimes K_1=U\otimes K_2$}&\multirow{2}{*}{$\mathcal{T}_{-\frac{7}{8}}\oplus\mathbb{F}_{-\frac{7}{8}}$}\\
&$[ -2, 0, 0, 0, 0, 0, 0, 3 ]$&$7/8$&&\\
\hline

\multirow{2}{*}{2}&$[-2,0,0,0,0,0,0,1]$&$7/8$&\multirow{2}{*}{$U\otimes K_1=U\otimes K_2$}&\multirow{2}{*}{$\mathcal{T}_{-\frac{7}{8}}\oplus\mathbb{F}_{-\frac{7}{8}}$}\\
&$[ 0, 0, 0, 0, 0, 0, 0, 3 ]$&$7/8$&&\\
\hline
\multirow{2}{*}{3}&$[-2,0,0,0,0,0,0,-1]$&$-1/8$&\multirow{2}{*}{$U\otimes K_1=U^2\otimes K_2$}&\multirow{2}{*}{$\mathcal{T}_{-\frac{15}{8}}\oplus\mathbb{F}_{\frac{1}{8}}$}\\
&$[ 0, 0, 0, 0, 0, 0, 0, 1 ]$&$15/8$&&\\
\hline
\multirow{2}{*}{4}&$[2,0,0,0,0,0,0,1]$&$-1/8$&\multirow{2}{*}{$U\otimes K_1=U^2\otimes K_2$}&\multirow{2}{*}{$\mathcal{T}_{-\frac{15}{8}}\oplus\mathbb{F}_{\frac{1}{8}}$}\\
&$[ 0, 0, 0, 0, 0, 0, 0, -1 ]$&$15/8$&&\\
\hline
\multirow{3}{*}{$5+j$}&$[ -2, 0, \cdots, \underbrace{2}_{j+2},\cdots, 0, -1 ]$&3/4&\multirow{3}{*}{$U\otimes K_1=U\otimes K_2$}&\multirow{3}{*}{$\mathcal{T}_{-\frac{3}{4}}\oplus\mathbb{F}_{-\frac{3}{4}}$}\\
&$[ 0, 0, \cdots, \underbrace{2}_{j+2},\cdots, 0, 1 ]$&3/4&&\\
&$j=0,\cdots,5$&&&\\
\hline
\end{tabular}
\end{center}
\vspace{0.3cm}

\begin{center}
\begin{tabular}{|c|c|}
\hline
$\mathrm{Spin}^c$&Root Vectors\\
\hline
\multirow{2}{*}{1}&$[ 2, 0, \cdots, 0, \underbrace{-4}_i, 0,\cdots, 0, -3 ]$\\
&$i=2,\cdots,7$\\
\hline
\multirow{3}{*}{2}&$[ 0, 0, 0, 0, 0, 0, 0, -5 ],$\\
&$[ 0, 0, \cdots, 0, \underbrace{-4}_i, 0,\cdots, 0, 1 ],$\\
&$i=2,\cdots,7$\\
\hline
\multirow{2}{*}{3}&$[ 0, 0, \cdots, 0, \underbrace{-4}_i, 0,\cdots, 0,- 1 ],$\\
&$i=2,\cdots,7$\\
\hline
$4$&$[ -2, 0, 0, 0, 0, 0, 0, -3 ]$\\
\hline
\multirow{3}{*}{$5+j$}&$[ 2,0, \cdots, \underbrace{-4}_i,\cdots,\underbrace{-2}_{j+2},\cdots, 0, -1 ]$\\
&$i=2,\cdots,7$\\
&$j=0,\cdots,5$\\
\hline

\end{tabular}
\end{center}

\caption{$HF^+$ of $\displaystyle M(-4,\frac{1}{2},\cdots,\frac{1}{2},\frac{1}{3})$ for $10$ $\mathrm{spin}^c$ structures.}
\label{fig:table1}
\end{table}

\vspace{0.2cm}

\section{Calculation of $\sigma$}\label{sec:sigma}

\vspace{0.2cm}

In this section we shall prove Theorem \ref{nesigma} by  calculating explicitly the $\sigma$ invariant of a family of contact $3$-manifolds. Our argument is based on a previous work of Rustamov, \cite{R1}. 

\vspace{0.3cm}

For every positive integer $n$, consider the contact manifold $(Y_n,\xi_n)$ obtained from $(S^3,\xi_{\mathrm{std}})$ by doing Legendrian surgery on the $(2,2n+1)$ torus knot $L_n$ stabilized $2n-1$ times as in  figure \ref{fig:legsymplectic2}. Observe that the Thurston-Bennequin invariant of $L_n$ is zero so the topological surgery coefficient is negative one. In fact, the $3$-manifold $Y_n$ is the Brieskorn homology sphere $\Sigma(2,2n+1,4n+3)$.

\vspace{0.3cm}

\begin{figure}[h]
\begin{center}
	\includegraphics[width=0.40\textwidth]{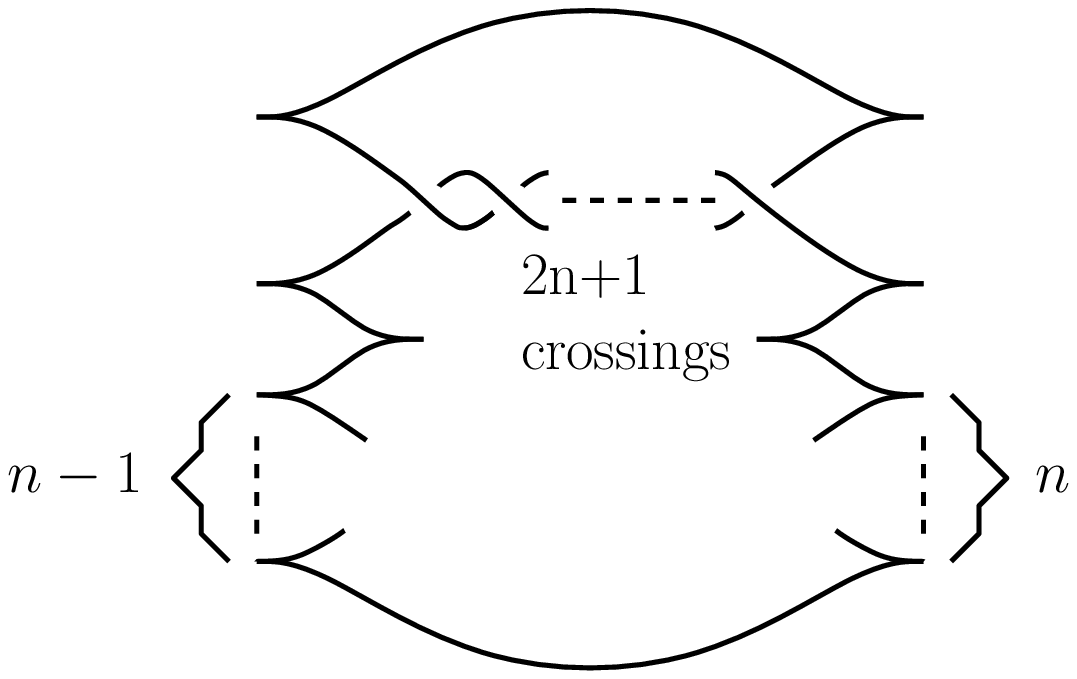}
	\caption{$K_n$}
	\label{fig:legsymplectic2}
	\end{center}
\end{figure}

\begin{theo}\label{sigma}
 $\sigma(Y_n,\xi_n)=-(p_n-1)$ where $p_n$ is the $n$th element of the sequence $$1,1,2,2,3,3,\cdots .$$ 

\end{theo}

\vspace{0.2cm}

Clearly, Theorem \ref{sigma} implies Theorem \ref{nesigma}. Another immediate application of Theorem \ref{sigma} is that $(Y_n,\xi_n)$ can not be supported by a planar open book.  This was first pointed out in \cite{OSS}. Finally, combining this theorem with the fact that $\sigma$ invariant respects the partial ordering coming from Stein cobordisms we have the following corollary.

\begin{cor}
There is no Stein cobordism from $(Y_n,\xi_n)$ to $(Y_m,\xi_m)$ if $n>m+1$. In particular, one can not obtain $(Y_m,\xi_m)$ from $(Y_n,\xi_n)$ via Thurston-Bennequin minus one $(tb-1)$ surgery on a Legendrian link.
\end{cor}

The above corollary should be compared to a classical result of Ding and Geiges in \cite{DG} where it was proved that any two contact manifolds can be obtained from each other via a sequence of $(tb-1)$ or $(tb+1)$ contact surgeries. In fact, one can always choose such a sequence which contains at most one $tb+1$ surgery. Therefore, the corollary tells us that the existence of $(tb+1)$ surgery is essential even though the contact manifolds in both ends are Stein fillable.

%\vspace{0.5cm}
\newpage 

\textbf{Proof of Theorem \ref{sigma}.}
Let $V$ be the $4$-manifold obtained by attaching a Weinstein $2$-handle to a $4$-ball along $L_n$. Eliashberg's theorem \cite{E} tells that $V$ admits a Stein structure.  Let $\mathfrak{s}$ be the canonical $\mathrm{spin}^c$ structure on $V$, and  denote the homology class determined by the $2$-handle (for some orientation of $L_n$) by $h$. The way we stabilize $L_n$ ensures that.

\begin{equation}\label{eq:chern1}
c_1(\mathfrak{s})(h)=\mathrm{rot}(L_n)=\pm 1
\end{equation}

\vspace{0.2cm}

Where $\mathrm{rot}(L_n)$ stands for the rotation number of $L_n$. Note that the sign of the rotation number depends on how we orient $L_n$. Next, $V$ is  blown-up $n+2$ times, and we do the handleslides indicated in figure \ref{fig:seqblowup}. We see that the resulting $4$-manifold is given by the plumbing graph $G$ in figure \ref{fig:config3}. The manifold $X(G)$ is no longer Stein but it does admit a symplectic structure.  Let $\mathfrak{s}'$ be the canonical $\mathrm{spin}^c$ structure on this symplectic manifold. Let $e_i$ denote the homology class of the $i$th exceptional sphere. We have 

\begin{equation}\label{eq:chern2}
c_1(\mathfrak{s}')(e_i)=1 \;\;\; i=1,2,\cdots,n+2.
\end{equation}

\vspace{0.2cm}

\begin{figure}[h]
\begin{center}
	\includegraphics[width=.35\textwidth]{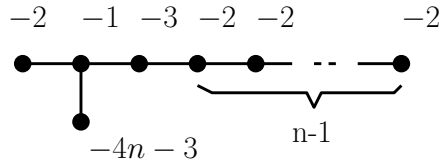}
	\caption{A Plumbing graph for Brieskorn homology sphere $\Sigma (2,2n+1,4n+3)$}
	\label{fig:config3}	
\end{center}
\end{figure}

Order the vertices of $G$ so that first four are the ones with weight $-1,-2,-3$ and $-4n-3$ respectively, and all the remaining ones corresponding to $-2$s on right are ordered according to the distance from the root starting with the closest one. In \cite{R1}, Rustamov proves that 

$$HF^+(-Y_n)=\mathcal{T}^+_{0}\oplus\mathbb{F}^{p_n}_{(0)} \oplus \bigoplus_{i=1}^{n-1}(\mathbb{F}^{p_i}_{q_{n-i}}\oplus \mathbb{F}^{p_i}_{q_{n-i}})$$

\noindent where $q_i=i(i+1)$ and $\mathbb{F}^r_{(k)}=\mathbb{F}[U]/U^r\mathbb{F}[U]$ and $U^{r-1}$ lies in degree $k$. More precisely, he shows that $\mathrm{Ker}U\subset HF^+(-Y_n)$ is generated by the characteristic vectors

$$K_i=(1,0,-1,-4n-3+2i,0,0,\cdots,0), \;\;\; i=1,2,\cdots,2n.$$

\noindent He also proves that the minimal relations are given as follows:

\begin{eqnarray}
U^{p_i}\otimes K_i& \sim & U^{p_i+q_{n-i}}\otimes K_{n+1} \\
U^{p_i}\otimes K_{n+i}& \sim &U^{p_i+q_{n-i}}\otimes K_{n} 
\end{eqnarray}

\vspace{0.2cm}

\noindent where $i=1,2,\cdots,n$. Note that the characteristic vectors $K_n$ and $K_{n+1}$ are in the bottom level, and their degree is zero.  

\vspace{0.4cm}

Our aim is to pin down the contact invariant $c^+(\xi_n)$ in $HF^+(-Y_n)$. Note that Proposition \ref{repr} in the stated form  can not be applied directly as it concerns Stein fillings of plumbed manifolds. However, as indicated in Remark \ref{rem:repr} it is also true for strong symplectic fillings. The only difference in the proof is that one uses Ghiggini's generalization \cite{Ghi} of Plamenevskaya's theorem \cite{P}. Alternatively, one can use the blow-up formula and handleslide invariance for this particular case to see that equations \ref{eq:rep1} and \ref{eq:rep2} hold. In any case, we see that the contact invariant $c^+(\xi_n)$ is represented by the first Chern class $c_1(\mathfrak{s}')$ of the canonical $\mathrm{spin}^c$ structure. In figure \ref{fig:seqblowup}, we keep track of the homology classes in order to pin down the first Chern class. By Equations \ref{eq:chern1} and \ref{eq:chern2}, we have $c_1(\mathfrak{s}')=K_n$ or $K_{n+1}$ depending on the orientation of  $L_n$, but the contact invariant is in the image of $U^{p_n-1}$ in any case.

\begin{flushright}
$\Box$
\end{flushright}

\vspace{2cm}

%\begin{figure}[p]
%	\includegraphics[width=1.00\textwidth]{table1.eps}
%	\caption{$HF^+$ of $\displaystyle M(-4,\frac{1}{2},\cdots,\frac{1}{2},\frac{1}{3})$ for $10$ $\mathrm{spin}^c$ structures. Characteristic vectors on left are the generators of $\mathrm{Ker}(U)$. Ones that lie in the same box induce the same $\mathrm{spin}^c$ structure on boundary.}
%	\label{fig:table1}
%\end{figure}

\begin{figure}[p]
	\includegraphics[width=1.00\textwidth]{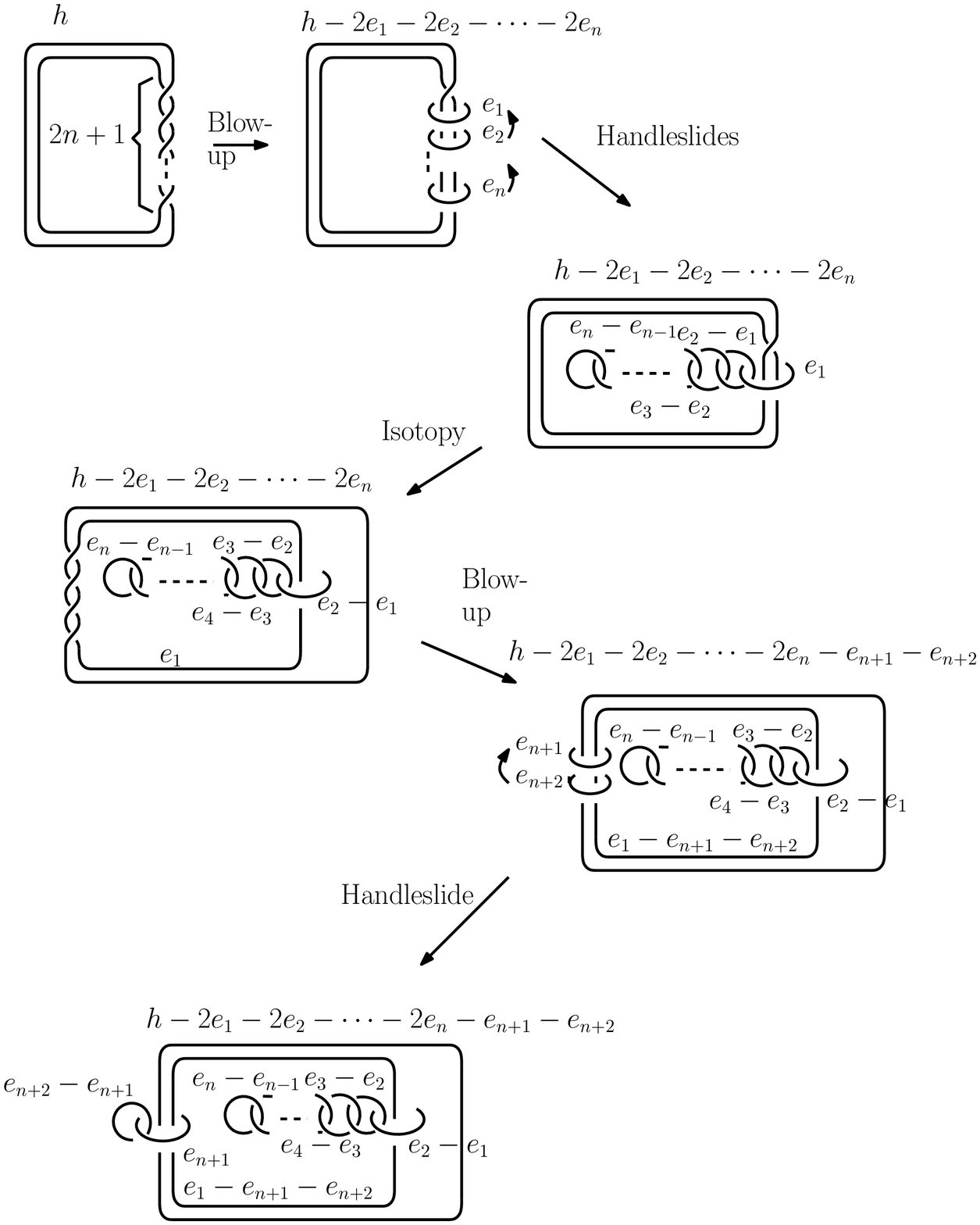}
	\caption{Sequence of Blow-ups from $K_n$ to  plumbing}
	\label{fig:seqblowup}
\end{figure}

\end{document}